\documentclass[preprint,12pt]{elsarticle1}

\usepackage{subfig}
\usepackage{graphicx}
\usepackage{caption,color}   % 24.62
\usepackage{amssymb}
\usepackage{amsthm}
\usepackage{amsmath}
\usepackage{epic}
\usepackage{setspace}
\usepackage{float}
\usepackage{multirow}

\usepackage{hyperref}
\usepackage{xcolor}
\usepackage{marginnote}

\newtheorem{thm}{Theorem}[section]
\newtheorem{cor}[thm]{Corollary}

\newtheorem{lem}[thm]{Lemma}
\newtheorem{pro}[thm]{Proposition}

\numberwithin{equation}{section}
\journal{}

\newcommand{\N}{\mathcal{N}}

\begin{document}
\begin{spacing}{1.15}
\begin{frontmatter}
\title{\textbf{
Spectra of power hypergraphs and signed graphs via parity-closed walks}}

\author[label1,label2]{Lixiang Chen}\ead{chenlixiang@hrbeu.edu.cn}
\author[label2]{Edwin R. van Dam}\ead{Edwin.vanDam@tilburguniversity.edu}
\author[label1]{Changjiang Bu}\ead{buchangjiang@hrbeu.edu.cn}

\address[label1]{College of Mathematical Sciences, Harbin Engineering University, Harbin, PR China}
\address[label2]{Department of Econometrics and O.R., Tilburg University, Tilburg, Netherlands}

\begin{abstract}
The $k$-power hypergraph $G^{(k)}$ is the $k$-uniform hypergraph that is obtained by adding $k-2$ new vertices to each edge of a graph $G$, for $k \geq 3$.
A parity-closed walk in $G$ is a closed walk that uses each edge an even number of times. In an earlier paper, we determined the eigenvalues of the adjacency tensor of $G^{(k)}$ using the eigenvalues of signed subgraphs of $G$. Here, we express the entire spectrum (that is, we determine all multiplicities and the characteristic polynomial) of $G^{(k)}$ in terms of parity-closed walks of $G$. Moreover, we give an explicit expression for the multiplicity of the spectral radius of $G^{(k)}$. Our results are mainly obtained by exploiting the so-called trace formula to determine the spectral moments of $G^{(k)}$. As a side result, we show that the number of parity-closed walks of given length is the corresponding spectral moment averaged over all signed graphs with underlying graph $G$. We also extrapolate the characteristic polynomial of $G^{(k)}$ to $k=2$, thereby introducing a pseudo-characteristic function. Among other results, we show that this function is the geometric mean of the characteristic polynomials of all signed graphs on $G$ and characterize when it is a polynomial.
This supplements a result by Godsil and Gutman that the arithmetic mean of the characteristic polynomials of all signed graphs on $G$ equals the matching polynomial of $G$.
\end{abstract}


\end{frontmatter}

\section{Introduction}

For a given graph $G=(V,E)$ and a sign function $\pi:E \rightarrow \{\pm1\}$, we denote by $\phi_{\pi}(\lambda)$ the characteristic polynomial of the signed adjacency matrix of a signed graph $G_\pi$ on $G$.
Let $\Pi$ denote the set of all sign functions on $E$.
Godsil and Gutman \cite{Godsil1978onthematching} proved the remarkable result that the arithmetic mean of the characteristic polynomials of all signed graphs with underlying graph $G$ equals the matching polynomial $\alpha(\lambda)$ of $G$, i.e.,
\begin{align*}
\alpha(\lambda)=2^{-|E|}\sum_{\pi \in \Pi}{\phi_{\pi}(\lambda)}.
\end{align*}
The result of Godsil and Gutman was used by Marcus, Spielman, and Srivastava \cite{Marcus2015Interlacing} to show the existence of an infinite family of $d$-regular bipartite Ramanujan graphs, for any $d\geq 3$.

The $k$-power hypergraph $G^{(k)}$ is the $k$-uniform hypergraph that is obtained by adding $k-2$ new vertices to each edge of a graph $G$, for $k \geq 3$ \cite{Hu2013Cored}.
A parity-closed walk in $G$ is a closed walk that uses each edge an even number of times.
In an earlier paper \cite{Chen2022All}, we determined all eigenvalues of the adjacency tensor of $G^{(k)}$ using the eigenvalues of signed subgraphs of $G$.

Here, we will extend combinatorial methods that are common for ordinary graphs to power hypergraphs in the sense that we will find a relation between the spectral moments of $G^{(k)}$ and counts of parity-closed walks in $G$. In order to determine these spectral moments, we will mainly exploit the so-called trace formula for tensors by Shao, Qi, and Hu \cite{shao2015some}.

In particular, we will determine the entire spectrum (and the characteristic polynomial) of $G^{(k)}$ by determining expressions for the multiplicities of the eigenvalues in terms of parity-closed walks of $G$. Moreover, we give an explicit expression for the multiplicity of the spectral radius. Note that from the Perron-Frobenius Theorem (for matrices), it is known that the spectral radius of a connected graph is an eigenvalue of the graph with multiplicity $1$. Part of the Perron-Frobenius Theorem has been generalized to tensors, in particular, it is known that the spectral radius of a uniform hypergraph is an eigenvalue \cite{chang2008perron}. However, it is unknown what its multiplicity is.

As a side result (that has nothing to do with hypergraphs), we show that the number of parity-closed walks of given length in a graph $G$ is the corresponding spectral moment averaged over all signed graphs with underlying graph $G$. We also extrapolate the characteristic polynomial of $G^{(k)}$ to $k=2$, thereby introducing a pseudo-characteristic function. Among other results, we show that this function is the geometric mean of the characteristic polynomials of all signed graphs on $G$ and characterize when it is a polynomial.
This supplements the result by Godsil and Gutman \cite{Godsil1978onthematching} that the arithmetic mean of the characteristic polynomials of all signed graphs on $G$ equals the matching polynomial of $G$.

\subsection{Organization of the paper}

The rest of this paper is organized as follows:
In Section \ref{jieshaozhangjie}, we introduce some notation and give some lemmas about the trace of tensors (Section \ref{sec:2.1}), the spectral moments of hypergraphs (Section \ref{sec2.2}), and the spectral radius of signed graphs (Section \ref{sec:2.3}).

In Section \ref{sec:pcwalkssigned}, we show that the number of parity-closed walks of given length in a graph $G$ is the arithmetic mean of spectral moments of all signed graphs on $G$. This is not a result about hypergraphs, nor eigenvalues, but it will be used later in the paper.

Section \ref{michaotupuju} is the main part that relates concepts in the power hypergraph $G^{(k)}$ to concepts in $G$. The section is divided into a part on Eulerian walks in related digraphs (Section \ref{sec:eulerian}), a part on the number of spanning trees (Section \ref{sec:spanningtrees}), and a final part (Section \ref{sec:coveringclosedwalks}), where we ultimately express the spectral moments of the power hypergraph in terms of counts of parity-closed walks in $G$ in Proposition \ref{zydl}.

In Section \ref{sec:charpoly}, eigenvalues and their multiplicities come into play. In Theorem \ref{pro5.1}, we determine expressions for the multiplicities and the characteristic polynomial of a power hypergraph $G^{(k)}$, all in terms of $k$, counts of parity-closed walks in $G$, and eigenvalues of signed subgraphs of $G$.

In Section \ref{sec4}, we extend the expressions for the characteristic polynomial of $G^{(k)}$ to the (hypothetical) case $k=2$ and study the obtained pseudo-characteristic function. In particular, in Theorem \ref{jieshao1} we show that it is the geometric mean of the characteristic polynomials of all signed graphs on $G$.

In Section \ref{sec5}, we finish the paper by determining the multiplicity of the spectral radius of a power hypergraph.

\section{Preliminaries}\label{jieshaozhangjie}

In this section, we introduce some basic notation and give auxiliary lemmas on the trace of tensors, the spectral moments of hypergraphs, and the spectral radius of signed graphs.

\subsection{The trace of tensors}\label{sec:2.1}
For a positive integer $n$, let $\left[ n \right] = \left\{ {1, \ldots ,n} \right\}$.
A $k$-order $n$-dimensional complex tensor $A= \left( {{t_{{i_1} \cdots {i_k}}}} \right) $ is a multidimensional array with $n^k$ entries in $\mathbb{C}$, where ${i_j} \in \left[ n \right]$, $j = 1, \ldots ,k$.
For $x = {\left( {{x_1}, \ldots ,{x_n}} \right)^{\top}} \in {\mathbb{C}^n}$, we let  ${x^{\left[ {k - 1} \right]}} = {\left( {x_1^{k - 1}, \ldots ,x_n^{k - 1}} \right)^{\top}}$. The $i$-th component of the vector $A{x^{k - 1}} \in \mathbb{C}^{n}$ is defined as
\[{\left( {A{x^{k - 1}}} \right)_i} = \sum\limits_{{i_2}, \ldots ,{i_k}=1}^n {{a_{i{i_2} \cdots {i_k}}}{x_{{i_2}}} \cdots {x_{{i_k}}}} .\]
If there exists a nonzero vector $x =(x_1,x_2,\ldots,x_n)^{\top}\in {\mathbb{C}^n}$ such that $A{x^{k - 1}} = \lambda {x^{\left[ {k - 1} \right]}}$, then $\lambda$ is called an \emph{eigenvalue} of $A$ and $x$ is an \emph{eigenvector} of $A$ corresponding to $\lambda$ \cite{lim2005singular,qi2005eigenvalues}.
The \emph{characteristic polynomial} of $A$ is defined as the resultant of the polynomial system $(\lambda x^{[k-1]}-Ax^{k-1})$ \cite{qi2005eigenvalues}.
The (algebraic) multiplicity of an eigenvalue is the multiplicity as a root of the characteristic polynomial.

Also the trace of a tensor can be defined such that it generalizes the trace of a matrix.
Shao, Qi, and Hu \cite{shao2015some} obtained a useful expression for the trace of general tensors. Using this, we will give an expression for the spectral moments of a hypergraph $H$ in terms of the number of connected subhypergraphs of $H$ in Section \ref{sec2.2}.

In order to describe the tensor trace formula of Shao et al.~\cite{shao2015some}, we introduce some related notation.
Let $\mathcal{F}_d=\{(i_1\alpha_1,\ldots,i_d\alpha_d):1\leq i_1\leq \cdots \leq i_d \leq n; \alpha_1, \ldots,\alpha_d \in [n]^{k-1}\}$.
Let $f=(i_1\alpha_1, \ldots, i_d\alpha_d) \in \mathcal{F}_d$, where $i_j\alpha_j \in [n]^k$, $j = 1, \ldots, d$.
For a $k$-order $n$-dimensional tensor $A=(a_{i_1\cdots i_k})$, let $\pi(f)=\prod_{j=1}^d{a_{i_j \alpha_j}}$. We will now construct a multi-digraph $D_f$, in which we let $(v_1,v_2)$ denote an arc from vertex $v_1$ to vertex $v_2$.
For $i_j\alpha_j=i_jv_1\ldots v_{k-1}$, we let the set of arcs from $i_j$ to $v_1, v_2, \ldots, v_{k-1}$ be $\theta(i_j\alpha_j)=\{(i_j, v_1), \ldots , (i_j, v_{k-1})\}$.
Let $\Theta(f)$ be the multi-set $\bigcup_{j=1}^d{\theta(i_j\alpha_j)}$.
Let $b(f)$ be the product of the factorials of the multiplicities of the arcs in $\Theta(f)$.
Let $c(f)$ be the product of the factorials of the outdegrees of all vertices in $\Theta(f)$.
Denote the set of all closed walks using all arcs in $\Theta(f)$ by $\mathcal{W}(f)$.

\begin{lem}\label{Shaodegonghsi}\cite[(3.5)]{shao2015some}
Let $A$ be a $k$-order $n$-dimensional tensor. Then
\begin{align*}
\mathrm{Tr}_d(A)=(k-1)^{n-1}\sum_{f \in \mathcal{F}_d}{\frac{b(f)}{c(f)}\pi_f(A)|\mathcal{W}(f)|}.
\end{align*}
\end{lem}

\subsection{The spectral moments of hypergraphs}\label{sec2.2}

A hypergraph $H=(V_H,E_H)$ is called \emph{$k$-uniform} if each edge of $H$ contains exactly $k$ vertices.
Similar to the relation between graphs and matrices, there is a natural correspondence between uniform hypergraphs and tensors.
For a $k$-uniform hypergraph $H$ with $n$ vertices, its \emph{adjacency tensor} ${A}_H=(a_{i_1i_2\ldots i_k})$ is a $k$-order $n$-dimensional tensor, where
\[{a_{{i_1}{i_2} \ldots {i_k}}} = \left\{ \begin{array}{l}
 \frac{1}{{\left( {k - 1} \right)!}},{\kern 37pt}\mathrm{ if}{\kern 2pt}{ \left\{ {{i_1},{i_2},\ldots ,{i_k}} \right\} \in {E_H}}, \\
 0, {\kern 57pt}\mathrm{ otherwise}. \\
 \end{array} \right.\]
When $k=2$, ${A}_H$ is the usual adjacency matrix of the graph $H$.
The characteristic polynomial of the adjacency tensor $A_H$ is called the characteristic polynomial of hypergraph $H$.
We note that expressions for the coefficients of the characteristic polynomial of uniform hypergraphs have been obtained, by Clark and Cooper \cite{CLARK20211}, but we will not use these.

It is known that the $d$-th order spectral moment of a graph $G$ is equal to the $d$-th order trace of the adjacency matrix $A_G$ \cite{cvetkovic1980spectra}.
Similar to the case of graphs, the $d$-th order spectral moment $\mathrm{S}_d(H)$  of a hypergraph $H$ (i.e., the sum of $d$-th powers of all eigenvalues of $H$) is equal to the $d$-th order trace $\mathrm{Tr}_d(A_H)$ of the adjacency tensor $A_H$ \cite{hu2013determinants}.

From Lemma \ref{Shaodegonghsi}, we have
\begin{align}\label{zhanglianggongshi}
\mathrm{S}_d(H)=\mathrm{Tr}_d(A_H)=(k-1)^{|V(H)|-1}\sum_{f \in \mathcal{F}_d}{\frac{b(f)}{c(f)}\pi_f(A_H)|\mathcal{W}(f)|}.
\end{align}
Without loss of generality, we only consider $f=(i_1\alpha_1, \ldots, i_d\alpha_d) \in \mathcal{F}_d$ for which $\pi_f(A_H)|W(f)| \neq 0$.
Let $i_j\alpha_j=i_jv^{(j)}_1v^{(j)}_2\ldots v^{(j)}_{k-1}$ for all $j \in [d]$.
Since $\pi_f(A_H) = \prod_{j=1}^d(A_H)_{i_j\alpha_j}$, we know that $e_j=\{i_j,v^{(j)}_1,v^{(j)}_2,\ldots, v^{(j)}_{k-1}\}$ is a hyperedge of $H$, i.e., $f$ consists of $d$ rooted hyperedges.
We construct a $k$-uniform hypergraph $H_f$ such that $V(H_f)=\bigcup_{j=1}^d{e_j}$ and $E(H_f)=\bigcup_{j=1}^d{\{e_j\}}$.
Obviously, $H_f$ is a subhypergraph of $H$.
Consider next a set of representatives of isomorphic such subhypergraphs (so-called motifs)
$$\mathcal{H}_d=\{\widehat{H}:H_f \cong \widehat{H} ~\mathrm{and}~ \pi_f(A_H)|W(f)| \neq 0 ~\mathrm{for}~ \mathrm{some} ~f \in \mathcal{F}_d\}.$$
For $\widehat{H} \in \mathcal{H}_d $, let $\mathcal{F}_d(\widehat{H})=\{f: f \in \mathcal{F}_d ~\mathrm{and} ~H_f \cong \widehat{H}\}$.
We use $N_H(\widehat{H})$ to denote the number of subhypergraphs of $H$ which are isomorphic to $\widehat{H}$.
By \eqref{zhanglianggongshi}, we have
\begin{align*}\label{yibanchaotugongshi}
\mathrm{S}_d(H)
&=(k-1)^{|V(H)|-1}((k-1)!)^{-d}\sum_{\widehat{H} \in \mathcal{H}_d}%\left(
{N_H(\widehat{H})\sum_{f \in \mathcal{F}_d(\widehat{H})}{\frac{b(f)}{c(f)}|\mathcal{W}(f)|}}.%\right).
\end{align*}
We call the coefficient
\begin{equation}\label{spectralcoefficient}
c_d(\widehat{H})=((k-1)!)^{-d}\sum_{f \in \mathcal{F}_d(\widehat{H})}{\frac{b(f)}{c(f)}|\mathcal{W}(f)|}
\end{equation}
the $d$-th order {\em spectral moment coefficient} of $\widehat{H}$.
We note that for a graph $G$, the coefficient $c_d(G)$ equals the number of covering closed walks of length $d$, that is, closed walks that use each edge at least once.
We thus give the following subgraph structure interpretation of the spectral moments of hypergraphs:
\begin{align}\label{yibanchaotupujugongshi}
\mathrm{S}_d(H)=(k-1)^{|V(H)|-1}\sum_{\widehat{H} \in \mathcal{H}_d}{c_d}(\widehat{H})N_{H}(\widehat{H}).
\end{align}
How to reduce and calculate the $d$-th order spectral moment coefficient
${c_d}(\widehat{H})$ is the key question in Section \ref{michaotupuju}. Here we make a first few steps in this direction.
For $f \in \mathcal{F}_d(\widehat{H})$, we constructed in Section \ref{sec:2.1} a  multi-digraph $D_f=\left(V(D_f),E(D_f)\right)$ with $V(D_f)=V(\widehat{H})$ and $E(D_f)=\Theta(f)$.
Since $\mathcal{W}(f)$ is the set of all closed walks with all arcs in $\Theta(f)$, we know that $|\mathcal{W}(f)| \neq 0$ if and only if $D_f$ is Eulerian.

Because isomorphic graphs $D_f$ contribute the same to the spectral moment coefficient in \eqref{spectralcoefficient}, we consider
the set of representatives of isomorphic Eulerian multi-digraphs
$$\mathfrak{D}_d(\widehat{H})=\{D: ~ D_f\cong D  ~ \mathrm{is} ~ \mathrm{Eulerian} ~ \mathrm{for}~\mathrm{some} ~f \in \mathcal{F}_d(\widehat{H})\}.$$
Then
\begin{align}\label{youxiangtu}
c_d(\widehat{H})=((k-1)!)^{-d}\sum_{D \in \mathfrak{D}_d(\widehat{H})}{|\{f:f \in \mathcal{F}_d(\widehat{H})~ \mathrm{and} ~ D_f\cong D \}|\frac{b(f)}{c(f)}|\mathcal{W}(f)|}.
\end{align}
In \eqref{youxiangtu}, $|\mathcal{W}(f)|$ is the number of the Eulerian walks in $D$.
An expression for the number of  the Eulerian walks in a multi-digraph  was  given by the so-called ``BEST Theorem" \cite{matthew2016multi}.

\begin{lem} \cite[Theorem 6]{matthew2016multi}\label{biwalkgongshi}
Let $D=(V(D),E(D))$ be a Eulerian multi-digraph. Let $b(D)$ denote the product of the factorials of the multiplicities of the arcs in $E(D)$. Then the number of Eulerian walks in $D$ is
\begin{align*}
\frac{|E(D)|}{b(D)}t(D)\prod_{v \in V(D)}{(\mathrm{deg}_D^+(v)-1)!},
\end{align*}
where $\mathrm{deg}_D^+(v)$ is the outdegree of vertex $v$ and $t(D)$ is the number of spanning trees of $D$.
\end{lem}

In \eqref{youxiangtu}, we have $b(f)=b(D)$ and $c(f)=\prod_{v \in V(D)}{\mathrm{deg}_D^+(v)!}$.
From Lemma \ref{biwalkgongshi}, we get the following expression for the $d$-th order spectral moment coefficient of a fixed hypergraph $\widehat{H}$.

\begin{lem}\label{pujuxishuyinli}
The $d$-th order spectral moment coefficient of a $k$-uniform hypergraph $\widehat{H}$ is
\begin{align*}
c_d(\widehat{H})=d(k-1)((k-1)!)^{-d}\sum_{D \in \mathfrak{D}_d(\widehat{H})}{\frac{\left|\{f:f \in \mathcal{F}_d(\widehat{H})~ \mathrm{and} ~ D_f\cong D \}\right|t(D)}{\prod_{v \in V(D)}\mathrm{deg}_D^+(v)}}.
\end{align*}
\end{lem}

We will use this result to further reduce the spectral moment coefficients of power hypergraphs in Section \ref{michaotupuju}.

The spectrum of a hypergraph is said to be $k$-symmetric if it is invariant under a rotation of an angle $2\pi/k$ in the complex plane \cite{cooper2012spectra}.
The spectrum of a $k$-power hypergraph is indeed $k$-symmetric \cite{hu2014eigenvectors}, and hence  $\mathrm{S}_d(G^{(k)}) = 0$ for $k \nmid d$. Note that Lemma \ref{xiangyi} below shows $k$-symmetry of the eigenvalues, but it does not say anything about the corresponding multiplicities.

From \eqref{yibanchaotupujugongshi}, we then have
\begin{align}\label{shizi5}
\mathrm{S}_d(G^{(k)})=\left\{ \begin{array}{l}(k-1)^{|V(G^{(k)})|-1}\sum_{\widehat{G} \in \mathcal{G}_d}{c_d}(\widehat{G}^{(k)})N_{G}(\widehat{G}),~~~~ k\mid d,\\
0,{\kern 207pt}k\nmid d,
\end{array} \right.
\end{align}
where $\mathcal{G}_d=\{\widehat{G}:\widehat{G}~\mathrm{is}~ \mathrm{a}~ \mathrm{subgraph} ~\mathrm{of}~ G ~\mathrm{and}~ \mathfrak{D}_d(\widehat{G}^{(k)}) \neq \emptyset\}$.

\subsection{Signed graphs and the spectral radius}\label{sec:2.3}

All the distinct eigenvalues of the power hypergraph $G^{(k)}$ are given by eigenvalues of signed subgraphs as follows \cite{Chen2022All}.

\begin{lem}\cite{Chen2022All}\label{xiangyi}
The complex number $\lambda$ is an eigenvalue of $G^{(k)}$ if and only if \\
\noindent$(a)$ some signed induced subgraph of $G$ has an eigenvalue $\sigma$ such that $\sigma^2=\lambda^k$, when $k=3$;\\
\noindent$(b)$ some signed subgraph of $G$ has an eigenvalue $\sigma$ such that $\sigma^2=\lambda^k$, when $k\geq4$.
\end{lem}

We call signed graphs $G_{\pi}$ and $G_{\pi'}$ switching equivalent if there is a diagonal matrix $D$ with diagonal entries $\pm 1$ such that $A(G_{\pi'})=D^{-1}A(G_{\pi})D$. Clearly, switching equivalent signed graphs have the same spectrum. The signed graphs $G_{+}$ and $G_{-}$ are the ones with all signs $+1$ and all signs $-1$, respectively.
If $G_{\pi}$ is switching equivalent to $G_{+}$, then $G_{\pi}$ is called a balanced signed graph. Note that $G_{+}$ and $G_{-}$ are switching equivalent if and only if $G$ is bipartite.

\begin{lem}\cite[Theorem 3.1]{stanic2019bounding}\label{stanic}
A signed graph $G_{\pi}$ contains a balanced spanning subgraph, say $H_{\tilde{\pi}}$, which satisfies $\lambda_{\mathrm{max}}(G_{\pi}) \leq \lambda_{\mathrm{max}}(H_{\tilde{\pi}})=\rho(H)$.
\end{lem}

We obtain the following result from Lemma \ref{stanic}.

\begin{lem}\label{yinlipubanjing}
Let $G$ be a connected graph and $\pi \in \Pi$. Then $\rho(G_{\pi})\leq \rho(G)$,
with equality if and only if $G_{\pi}$ is switching equivalent to $G_{+}$ or $G_{-}$.
Moreover, $G_{\pi}$ has an eigenvalue $\rho(G)$ if and only if it is switching equivalent to $G_{+}$, and it has an eigenvalue $-\rho(G)$ if and only if it is switching equivalent to $G_{-}$.
\end{lem}

\begin{proof}
Note that $\rho(G_{\pi})=\max\{\lambda_{\max}(G_{\pi}),\lambda_{\max}(G_{-\pi})\}$.
It is clear that if $G_{\pi}$ is switching equivalent to $G_{+}$, then it has an eigenvalue $\rho(G)$, and if it is switching equivalent to $G_{-}$, then it has an eigenvalue $-\rho(G)$, hence in both cases $\rho(G_{\pi})= \rho(G)$.

From Lemma \ref{stanic}, we have that $\rho(G_{\pi})=\lambda_{\max}(G_{\pm\pi}) \leq \rho(H) \leq \rho(G)$. If equality holds, then $H=G$, and hence $G_{\pm \pi}=H_{\tilde{\pi}}$, which is balanced, i.e., $G_{\pi}$ is switching equivalent to $G_{+}$ or $G_{-}$.

If $G_{\pi}$ has an eigenvalue $\rho(G)$, and it would not be switching equivalent to $G_{+}$, then by the above it must be switching equivalent to $G_{-}$, in which case it also has an eigenvalue $-\rho(G)$. But then $G$ is bipartite, in which case $G_{+}$ is switching equivalent to $G_{-}$, and hence to $G_{\pi}$ after all. Similarly it follows that if $G_{\pi}$ has an eigenvalue $-\rho(G)$, then it is switching equivalent to $G_{-}$.
\end{proof}

\section{The number of parity-closed walks via spectral moments of signed graphs}\label{sec:pcwalkssigned}
Our first goal is to show that the number of parity-closed walks $P_d$ of length $d$ in a graph $G$ is the arithmetic mean of spectral moments of all signed graphs with underlying graph $G$. Although this result does not concern hypergraphs, it will become relevant when we apply the results on hypergraphs in Section \ref{sec4}.

Recall that $\Pi$  denotes the set of all sign functions on the edge set $E$ of $G$.

\begin{thm}\label{xindingli1}
Let $G$ be a graph.
Then
\begin{align*}
{P}_d=2^{-|E|}\sum_{\pi \in \Pi}{\mathrm{S}_d(G_{\pi})}.
\end{align*}
\end{thm}

\begin{proof}
We denote the edges of $G$ by $e_1, e_2, \ldots, e_m$, where $m=|E|$.
Let $\mathcal{A}=\mathcal{A}(a_1,a_2,\ldots,a_m)$ denote the variable adjacency matrix of $G$ with entries
\begin{align*}
\mathcal{A}_{uv}=\left\{ \begin{array}{l}a_i,~~~~ \mathrm{if} ~\{u,v\}=e_i,\\
0,~~~~~\mathrm{otherwise},
\end{array} \right.
\end{align*}
where $a_i$ is a variable for each $i \in [m]$.
Let $\mathcal{W}_d$ be the set of all closed walks of length $d$ in $G$.
For $w \in \mathcal{W}_d$, let $t_i(w)$ be the number of times edge $e_i$ is used in the closed walk $w$.
The $d$-th order trace of the variable adjacency matrix $\mathcal{A}$ is a homogeneous polynomial of degree $d$ with respect to $a_1,a_2,\ldots,a_m$, i.e.,
\begin{align*}
  \mathrm{trace}(\mathcal{A}^d)=\sum_{w \in \mathcal{W}_d}{\prod_{e_i \in E(w)}{a_i^{t_i(w)}}}.
\end{align*}
If $t_i(w)$ is even for every $a_i \in E(w)$, then the closed walk $w$ is a parity-closed walk in $G$.
In order to account for parity-closed walks only, we will remove all monomials containing variables of odd degree from $\mathrm{trace}(\mathcal{A}^d)$. We then obtain the number of such walks by substituting $a_1=a_2=\cdots = a_m=1$.

Let $f_{a_i}$ be the operation on a polynomial $p$ such that
$$f_{a_i}\circ p(a_1,a_2,\ldots,a_m) = \frac{1}{2}p(a_1,\ldots, a_{i},\ldots,a_m)+\frac{1}{2}p(a_1,\ldots, -a_{i},\ldots,a_m).$$
It follows that $f_{a_i}\circ p$ is the polynomial obtained by removing all terms for which the degree of $a_i$ is odd in  the polynomial $p$. We therefore have to use $m$ such operations, and observe that
\begin{align*}
f_{a_m}\circ f_{a_{m-1}}\circ \cdots \circ f_{a_1}\circ p(a_1,a_2,\ldots,a_m) =
2^{-m}\sum_{\pi \in \Pi}p(\pi(e_1)a_1,\pi(e_2)a_2,\ldots,\pi(e_m)a_m)
\end{align*}

Thus, the number of parity-closed walks of length $d$ equals
\begin{align*}
{P}_d&=\left(f_{a_m}\circ f_{a_{m-1}}\circ \cdots \circ f_{a_1}\circ \mathrm{trace}(\mathcal{A}^d)\right)|_{a_1=a_2=\cdots=a_m=1}\\
  &=2^{-|E|}\sum_{\pi \in \Pi}{\mathrm{trace}(A(G_{\pi})^{d})}
=2^{-|E|}\sum_{\pi \in \Pi}{\mathrm{S}_d(G_{\pi})}.
\end{align*}
\end{proof}

\section{The spectral moments of power hypergraphs via the number of parity-closed walks}\label{michaotupuju}

Let $\mathcal{G}(m)$ be the set of all connected unlabeled subgraphs (so-called motifs) of $G$ with at most $m$ edges.
For ${\widehat{G} \in \mathcal{G}(m)}$, let $N_G(\widehat{G})$ denote the number of subgraphs of $G$ isomorphic to $\widehat{G}$. A closed walk in $\widehat{G}$ is called {\em covering} if it uses each edge at least once.
We use $p_{d}({\widehat{G}})$ to denote the number of covering parity-closed walks of length $d$ in ${\widehat{G}}$.
Define a function
\begin{align}\label{shizi1.1}
{\mathcal{S}}_d(k)=\left\{ \begin{array}{l}(k-1)^{|V|+|E|(k-2)-1}\sum\limits_{\widehat{G} \in \mathcal{G}(\frac{d}{k})}\frac{2^{|E(\widehat{G})|-|V(\widehat{G})|} k^{|E(\widehat{G})|(k-3)+|V(\widehat{G})|}}{(k-1)^{|V(\widehat{G})|+|E(\widehat{G})|(k-2)-1}}p_{\frac{2 d}{k}}(\widehat{G})N_{G}(\widehat{G}),k\mid d,\\
0,{\kern 328pt}k\nmid d,
\end{array} \right.
\end{align}
that depends on $G$, $d$, $k$ and involves parity-closed walks.
By substituting $k=2$ in \eqref{shizi1.1}, it follows that
\begin{align}\label{eq:Pd=S2}
\mathcal{S}_d(2)=\sum_{\widehat{G} \in \mathcal{G}(\frac{d}{2})}p_{d}(\widehat{G})N_{G}(\widehat{G})=P_d,
\end{align}
so the definition $S_d(k)$ extends a natural decomposition of counting parity-closed walks in $G$ to larger $k$.

The goal of this section is to show that ${\mathcal{S}}_d(k)$ is an expression for the spectral moments of the power hypergraph $G^{(k)}$ for $k \geq 3$. In order to accomplish this, we first need to look closer at the Eulerian digraphs $D_f$ and their numbers of spanning trees.

\subsection{Eulerian digraphs}\label{sec:eulerian}

A vertex $v \in V({G}^{(k)}) \setminus V({G})$ is called a {\em core vertex} of $G^{(k)}$.
For $\{i,j\} \in E(G)$, we use $\{i,j\}^{(k)}$ to denote the hyperedge of $G^{(k)}$ formed by adding $k-2$ core vertices to $\{i,j\}$. We denote the set of core vertices in $\{i,j\}^{(k)}$ by $\mathcal{N}_{ij}$.

We provide the following lemma to give the multiplicity of every arc of the multi-digraph $D$ in ${ \mathfrak{D}_{d}(\widehat{G}^{(k)})}$.
This will allow us to find all multi-digraphs ${D \in \mathfrak{D}_{d}(\widehat{G}^{(k)})}$ to compute the spectral moment coefficient $c_d(\widehat{G}^{(k)})$ from Lemma \ref{pujuxishuyinli}.

\begin{lem}\label{chongshuyinli}
Let $k \geq3$.
For ${D \in \mathfrak{D}_{d}(\widehat{G}^{(k)})}$, let $m_D(v,u)$ denote the multiplicity of the arc $(v,u)$ in $E(D)$. Let $\{i,j\} \in E(\widehat{G})$. Then $m_D(v,u)=m_D(v,u')$ for any three distinct vertices $v,u,u'$ in $\{i,j\}^{(k)}$ and
$m_D(i,j)+m_D(j,i)=2m_D(v,i)=2m_D(v,j)$ for every core vertex $v$ of $\{i,j\}^{(k)}$.
\end{lem}

\begin{proof}
From the construction of $D_f \cong D$, it follows that $m_D(v,u)=m_D(v,u')$ for any three distinct vertices $v,u,u'$ in $\{i,j\}^{(k)}$.

Because a core vertex $v \in \mathcal{N}_{ij}$ occurs in only one hyperedge, if follows that $\mathrm{deg}_D^+(v)=\sum_{u \in \{i,j\}^{(k)}\setminus \{v\}}{m_D(v,u)}=(k-1)m_D(v,i)$ and
\begin{align*}
\mathrm{deg}_D^-(v)&=\sum_{u \in \{i,j\}^{(k)}\setminus \{v\}}{m_D(u,v)}\\
&=m_D(i,v)+m_D(j,v)+\sum_{u \in \mathcal{N}_{ij}\setminus \{v\}}{m_D(u,v)}\\
&=m_D(i,j)+m_D(j,i)+\sum_{u \in \mathcal{N}_{ij}\setminus \{v\}}{m_D(u,i)}.
\end{align*}
Since $D \in \mathfrak{D}_d(\widehat{G}^{(k)})$ is Eulerian, we have $\mathrm{deg}_D^-(v)=\mathrm{deg}_D^+(v)$.
It yields that
\begin{align}\label{shizi1}
km_D(v,i)=m_D(i,j)+m_D(j,i)+\sum_{u \in \mathcal{N}_{ij}}{m_D(u,i)}
\end{align}
for all  $v \in \mathcal{N}_{ij}$.
It follows that $m_D(v,i)=m_D(u,i)$ for any $v, u \in \mathcal{N}_{ij}$.
By \eqref{shizi1}, we have $m_D(i,j)+m_D(j,i)=2m_D(v,i)=2m_D(v,j)$ for all $v \in \mathcal{N}_{ij}$.
\end{proof}

For ${D \in \mathfrak{D}_{d}(\widehat{G}^{(k)})}$,
let $D^*$ be the multi-digraph obtained by removing all core vertices from $D$ and let
$\mathfrak{D}^*=\mathfrak{D}^*_d(\widehat{G}^{(k)})=\{D^*:D \in \mathfrak{D}_{d}(\widehat{G}^{(k)})\}$. It follows from Lemma \ref{chongshuyinli} that there is a one-one correspondence between $\mathfrak{D}_d(\widehat{G}^{(k)})$ and $\mathfrak{D}^*_d(\widehat{G}^{(k)})$, i.e., $D$ can be reconstructed from $D^*$ using the derived equations in the statement of the lemma.

\begin{lem}\label{chongshuyinliEulerian}
Let $k \geq3$
and ${D \in \mathfrak{D}_{d}(\widehat{G}^{(k)})}$.
Then $D^*$ is Eulerian.
\end{lem}

\begin{proof}
We use Lemma \ref{chongshuyinli}. For any $i \in V(\widehat{G})$, we have that $$\mathrm{deg}_D^+(i)=\sum_{j:\{i,j\}\in E(\widehat{G})}{(k-1)m_D(i,j)}$$ and
\begin{align*}
\mathrm{deg}_D^-(i)&=\sum_{j:\{i,j\}\in E(\widehat{G})}{\left(m_D(j,i)+\sum_{v\in \mathcal{N}_{ij}}m_D(v,i)\right)}\\
&=\sum_{j:\{i,j\}\in E(\widehat{G})}{\left(m_D(j,i)+(k-2)\frac{m_D(i,j)+m_D(j,i)}{2}\right).}
\end{align*}
Since $D$ is Eulerian, we have $\mathrm{deg}_D^+(i)=\mathrm{deg}_D^-(i)$,
which implies that
$$\sum_{j:\{i,j\}\in E(\widehat{G})}{m_D(i,j)}=\sum_{j:\{i,j\}\in E(\widehat{G})}{m_D(j,i)},$$ i.e., $\mathrm{deg}_{D^*}^+(i)=\mathrm{deg}_{D^*}^-(i)$ for $i \in V(\widehat{G})=V(D^*)$.
Thus, $D^*$ is Eulerian.
\end{proof}

Using Lemma \ref{chongshuyinli}, we can now give an intuitive description of the set $\mathcal{G}_d$ from \eqref{shizi5}.
Recall that $\mathcal{G}(m)$ is the set of all connected unlabeled subgraphs (motifs) of $G=(V,E)$ with at most $m$ edges.

\begin{lem}\label{jihedingli}
Let $k \geq 3$. Then $\mathcal{G}_d=\mathcal{G}(\frac{d}{k})$ and $|E(D^*)|=\frac{2d}{k}$ for $D \in \mathfrak{D}_{d}(\widehat{G}^{(k)})$.
\end{lem}

\begin{proof}

We will first prove that $\mathcal{G}_d \subseteq \mathcal{G}(\frac{d}{k})$.

For any $\widehat{G} \in \mathcal{G}_d$ and $D \in \mathfrak{D}_d(\widehat{G}^{(k)})$, first note that $|E(D)|=d(k-1)$.
On the other hand, using Lemma \ref{chongshuyinli}, we have that
\begin{align*}
&|E(D)|
=\sum_{\{i,j\} \in E(\widehat{G})}{\sum_{u,v \in \{i,j\}^{(k)}\atop  u \neq v}{m_D(u,v)}}\notag \\
&=\sum_{\{i,j\} \in E(\widehat{G})}\left({\sum_{ v \in \{i,j\}^{(k)}\setminus \{i\} }m_D(i,v)}+{\sum_{  v \in \{i,j\}^{(k)}\setminus \{j\}}m_D(j,v)}+{\sum_{u \in \N_{ij}}\sum_{ v \in \{i,j\}^{(k)}\setminus \{u\}}m_D(u,v)}\right) \notag\\
&=\sum_{\{i,j\} \in E(\widehat{G})}\left({(k-1)m_D(i,j)}+{(k-1)m_D(j,i)}+{\sum_{u \in \N_{ij}}\sum_{v \in \{i,j\}^{(k)}\setminus \{u\}}m_D(u,v)}\right).
\end{align*}
Because $m_D(u,v)=\frac{m_D(i,j)+m_D(j,i)}{2}$ for any $u \in \N_{ij}$ and any ${v \in \{i,j\}^{(k)}} \setminus \{u\}$, it now readily follows that
$$|E(D)|=k(k-1)\sum_{\{i,j\} \in E(\widehat{G})}{\frac{m_D(i,j)+m_D(j,i)}{2}}.$$
Since $|E(D)|=d(k-1)$, we thus have that
\begin{align}\label{countedges}
\sum_{\{i,j\} \in E(\widehat{G})}{\frac{m_D(i,j)+m_D(j,i)}{2}}= \frac{d}{k}.
\end{align}

From Lemma \ref{chongshuyinli}, we know that ${\frac{m_D(i,j)+m_D(j,i)}{2}}$ is a positive integer for any $\{i,j\} \in E(\widehat{G})$.
Hence we obtain that $|E(\widehat{G})| \leq \frac{d}{k}$, i.e, $\widehat{G} \in \mathcal{G}(\frac{d}{k})$.

From \eqref{countedges} it also follows that
\begin{align*}|E(D^*)|&=\sum_{i \in V(\widehat{G})}\deg_{D^*}(i)=\sum_{i \in V(\widehat{G})} \sum_{j:\{i,j\} \in E(\widehat{G})} m_D(i,j)\\
 &= \sum_{\{i,j\} \in E(\widehat{G})}{m_D(i,j)+m_D(j,i)}=\frac{2d}{k}.
 \end{align*}

Finally, we will show that $\mathcal{G}_d  \supseteq \mathcal{G}(\frac{d}{k})$.
For any  $\widehat{G} \in \mathcal{G}(\frac{d}{k})$, there exists a closed walk $w$ in $\widehat{G}$ with length $\frac{2d}{k}$ such that all edges of $\widehat{G}$ are used an even number of times.
We first construct the multi-digraph $D_{w}=(V(\widehat{G}),E(w))$ without core vertices from the closed walk $w$. Then $D_{w}$ is clearly Eulerian.
Recall that $m_{D_{w}}(i,j)$ denotes the multiplicity of the arc $(i,j)$ in $E(w)$, and note that $\frac{m_{D_{w}}(i,j)+m_{D_{w}}(j,i)}{2}$ is a positive integer.
Next, we will add core vertices  and their incident arcs to $D_{w}$ in order to obtain a multi-digraph $D$.
For any $\{i,j\} \in E(\widehat{G})$ and any core vertex $v \in \N_{ij}$,
we construct $D$ such that $m_D(v,u)=\frac{m_{D_{w}}(i,j)+m_{D_{w}}(j,i)}{2}$ for any $u \in \{i,j\}^{(k)} \setminus \{v\}$ and $m_D(i,v)=m_D(i,j)=m_{D_{w}}(i,j)$.
It is easy to see that $D$ is Eulerian and $D \in\mathfrak{D}_d(\widehat{G}^{(k)})$, i.e, $\mathfrak{D}_d(\widehat{G}^{(k)})$ is not empty, since $D_{w}$ is Eulerian.
Thus, $\widehat{G} \in \mathcal{G}_d$.
\end{proof}

An additional consequence of the above proof is that every covering parity-closed walk $w$ gives rise to a Eulerian walk in some $D^{*}=D_w \in \mathfrak{D}^*_d(\widehat{G}^{(k)})=\mathfrak{D}^*$. This leads to the following result.

\begin{lem}\label{lem:parityclosedeulerian}
Let $k \geq 3$ and $\widehat{G} \in \mathcal{G}(\frac{d}{k})$. Then
$$p_{2d/k}(\widehat{G})=\sum_{D^* \in \mathfrak{D}^* }\frac{|E(D^*)|}{b(D^*)}\prod_{v \in V(D^*)}(\mathrm{deg}_{D^*}^+(v)-1)!t(D^*).$$
\end{lem}

\begin{proof} As remarked, every covering parity-closed walk $w$ gives rise to a Eulerian walk in some $D^{*}$. On the other hand, consider a Eulerian walk in some $D^* \in \mathfrak{D}^*$. First, note that the walk gives rise to a covering closed walk in the underlying graph $\widehat{G}$. Next, let $D$ be the corresponding multi-digraph in $\mathfrak{D}_d(\widehat{G}^{(k)})$. Then $m_D(i,j)=m_{D^*}(i,j)$ for any two non-core vertices $i,j$. By Lemma \ref{chongshuyinli}, we have $m_D(i,j)+m_D(j,i)=2m_D(v,i)$ for every core vertex $v$ of $\{i,j\}^{(k)}$, and hence $m_{D^*}(i,j)+m_{D^*}(j,i)$ is even. This implies that every edge in the corresponding walk is covered an even number of times, and hence the (covering) walk is parity-closed. The result now follows by Lemma
\ref{biwalkgongshi} (the ``BEST theorem").
\end{proof}

\subsection{The number of spanning trees}\label{sec:spanningtrees}

In order to obtain the spectral moments of $G^{(k)}$ from \eqref{shizi5}, we need to further reduce the spectral moment coefficients $c_d(\widehat{G}^{(k)})$, and hence we have to reduce the number of spanning trees $t(D)$ from Lemma \ref{pujuxishuyinli}.
By the Matrix-Tree theorem \cite{duval2009simplicial} and an expression for the determinant involving the Schur complement \cite{brualdi1983determinantal}, the number of spanning trees of a larger graph can be reduced to the number of spanning trees of a smaller weighted graph \cite{crabtree1966applications,fan1960note,griffing2014structural,zhou2021enumeration}.
Using a similar trick, we reduce the number of spanning trees $t(D)$ of the digraph $D$ to $t(D^*)$ as follows.

\begin{lem}\label{scsyinli}
Let $k \geq3$ and ${D \in \mathfrak{D}_{d}(\widehat{G}^{(k)})}$. Then
\begin{align}\label{scsshizi}
t(D)=t(D^*)k^{|E(\widehat{G})|(k-3)+|V(\widehat{G})|-1}2^{|E(\widehat{G})|-V(\widehat{G})+1}\prod_{\{i,j\} \in E(\widehat{G})}{\left(\frac{m_{D^*}(i,j)+m_{D^*}(j,i)}{2}\right)}^{k-2}.
\end{align}
\end{lem}

\begin{proof}
We write the Laplacian matrix of $D$ as a block matrix
$${L_D} =
\begin{bmatrix} M&N\\
P&Q
\end{bmatrix}.$$
The matrix $M$ is a $|V(\widehat{G})| \times |V(\widehat{G})|$ matrix, with $M_{ii}=\mathrm{deg}_D^+(i)$ and $M_{ij}=-m_D(i,j)$ for $i \in V(\widehat{G})$.

Denote the edges of $G$ by $e_1,e_2,\ldots,e_{|E(\widehat{G})|}$.
The matrix $Q$ is a block diagonal matrix with diagonal blocks $Q_{mm}=\frac{m_D(i,j)+m_D(j,i)}{2}(kI-J)$, where $\{i,j\}=e_m$, $I$ is an identity matrix and $J$ is an all-ones matrix of size $(k-2) \times (k-2)$.

The matrix $N$ is a block matrix with $|V(\widehat{G})| \times |E(\widehat{G})|$ blocks as follows.
Let $\mathbf{1}$ be the all-ones column vector of size $(k-2)$.
Then $N_{im}=-m_D(i,j)\mathbf{1}^{\top}$ if $e_m=\{i,j\}$ and $N_{im}=0$ if $i \notin e_m$.
Similarly, the matrix $P$ is a block matrix with blocks $P_{mi}=-\frac{m_D(i,j)+m_D(j,i)}{2}\mathbf{1}$ if $e_m=\{i,j\}$ and $P_{mi}=0$ if $i \notin e_m$.

Let $\widehat{L_D}$ denote the submatrix of $L_D$ obtained by deleting the first row and column, then we can write it as block matrix
$$\widehat{L_D} =
\begin{bmatrix} \widehat{M}&\widehat{N}\\
\widehat{P}&Q
\end{bmatrix}.$$

From the Matrix-tree Theorem \cite{duval2009simplicial}, we have $t(D)=\mathrm{det}(\widehat{L_D})$.
Note that the matrix $Q$ is invertible and $(Q^{-1})_{mm}=(Q_{mm})^{-1}=k^{-1}(m_D(i,j)+m_D(j,i))^{-1}(2I+J)$, where $\{i,j\}=e_m$.
From the determinant formula involving the Schur complement \cite{brualdi1983determinantal}, we have
\begin{align}\label{shizi3}
t(D)=\mathrm{det}(\widehat{L_D})=\mathrm{det}(Q)\mathrm{det}(\widehat{M}-\widehat{N}Q^{-1}\widehat{P}).
\end{align}
Furthermore,
\begin{align*}
{\left( {\widehat{M} - \widehat{N}{Q^{ - 1}}\widehat{P}} \right)_{ii}} &= {\widehat{M}_{ii}} - \sum\limits_{m = 1}^{\left| {E(\widehat{G})} \right|} {{\widehat{N}_{im}}{{\left( {{Q^{ - 1}}} \right)}_{mm}}{\widehat{P}_{mi}}} \\
&=\mathrm{deg}_D^+(i)-\sum_{j:\{i,j\} \in E(\widehat{G})}{(2k)^{-1}m_D(i,j)\mathbf{1}^\top (2I+J)\mathbf{1}}\\
&=\mathrm{deg}_D^+(i)-\frac{k-2}{2}\sum_{j:\{i,j\} \in E(\widehat{G})}{m_D(i,j)}\\
&=(k-1)\mathrm{deg}_{D^*}^+(i)-\frac{k-2}{2}\mathrm{deg}_{D^*}^+(i)
=\frac{k}{2}\mathrm{deg}_{D^*}^+(i)
\end{align*}
and
\begin{align*}
{\left( {\widehat{M} - \widehat{N}{Q^{ - 1}}\widehat{P}} \right)_{ij}} &= {\widehat{M}_{ij}} - \sum\limits_{m = 1}^{\left| {E(\widehat{G})} \right|} {{\widehat{N}_{im}}{{\left( {{Q^{ - 1}}} \right)}_{mm}}{\widehat{P}_{mj}}} \\
&=-m_D(i,j)-\frac{k-2}{2}m_D(i,j)
=-\frac{k}{2}m_{D^*}(i,j).
\end{align*}
Thus, for the Schur complement, we have ${\widehat{M} - \widehat{N}{Q^{ - 1}}\widehat{P}}=\frac{k}{2}\widehat{L_{D^*}}$, where $\widehat{L_{D^*}}$ is obtained from the Laplacian matrix of $D^*$ by removing the same row and column as in $M$.
Hence $\mathrm{det}(\widehat{M}-\widehat{N}Q^{-1}\widehat{P})=\mathrm{det}(\frac{k}{2}\widehat{L_{D^*}})=(\frac{k}{2})^{|V(\widehat{G})|-1}t(D^*)$.

Note also that if $e_m=\{i,j\}$, then
\begin{align*}
 \mathrm{det}(Q_{mm})&={\left(\frac{m_{D^*}(i,j)+m_{D^*}(j,i)}{2}\right)}^{k-2}\mathrm{det}(kI-J)\\
&=2k^{k-3}{\left(\frac{m_{D^*}(i,j)+m_{D^*}(j,i)}{2}\right)}^{k-2}.
\end{align*}
By \eqref{shizi3}, we thus have that
\begin{align*}
t(D)&=(\frac{k}{2})^{|V(\widehat{G})|-1}t(D^*)\mathrm{det}(Q)
=(\frac{k}{2})^{|V(\widehat{G})|-1}t(D^*)\prod_{m=1}^{|E(\widehat{G})|}{\mathrm{det}(Q_{mm})}\\
&=t(D^*)k^{|E(\widehat{G})|(k-3)+|V(\widehat{G})|-1}2^{|E(\widehat{G})|-V(\widehat{G})+1}\prod_{\{i,j\} \in E(\widehat{G})}{\left(\frac{m_{D^*}(i,j)+m_{D^*}(j,i)}{2}\right)}^{k-2}.
\end{align*}
\end{proof}

\subsection{The number of covering closed walks}\label{sec:coveringclosedwalks}

In a related paper \cite{Chen2020Spectral}, the spectral moment coefficients of power hypertrees were reduced to the spectral moment coefficients of trees as follows.
\begin{lem}\cite{Chen2020Spectral}\label{shudepujuxishuhuajiangongshi}
Let $k\geq3$ and let $\widehat{T}$ be a a tree.
Then
$$c_{\ell k}(\widehat{T}^{(k)})=\frac{ k^{|E(\widehat{T})|(k-2)+1}}{2(k-1)^{|V(\widehat{T})|+|E(\widehat{T})|(k-2)-1}}c_{2\ell}(\widehat{T}).$$
\end{lem}

Note that $c_{2\ell}(\widehat{T})$ is the number of covering closed walks of length $2\ell$ in $\widehat{T}$. We recall that a closed walk in which every edge is used an even number of times is called a parity-closed walk and that we denote the number of covering parity-closed walks of length $2\ell$ in $\widehat{G}$ by $p_{2\ell}(\widehat{G})$.
It is clear that for trees, every closed walk is parity-closed, so $c_{2\ell}(\widehat{T})= p_{2\ell}(\widehat{T})$.

Lemma \ref{shudepujuxishuhuajiangongshi} does not hold in the case of general graphs, but using the earlier results in this section, it can be generalized as follows.

\begin{lem}\label{pujuxishudingli}
Let $k\geq3$ and let $\widehat{G} \in \mathcal{G}(\ell)$. Then
\begin{align}\label{czkpzk}
c_{\ell k}(\widehat{G}^{(k)})=\frac{2^{|E(\widehat{G})|-|V(\widehat{G})|} k^{|E(\widehat{G})|(k-3)+|V(\widehat{G})|}}{(k-1)^{|V(\widehat{G})|+|E(\widehat{G})|(k-2)-1}}p_{2\ell}(\widehat{G}).
\end{align}
\end{lem}
\begin{proof}
From the spectral moment coefficient in Lemma \ref{pujuxishuyinli}, we have
\begin{align}\label{shizi7}
c_{\ell k}(\widehat{G}^{(k)})=\ell k(k-1)((k-1)!)^{-\ell k}\sum_{D \in \mathfrak{D}_{\ell k}(\widehat{G}^{(k)})}{\frac{\left|\{f:f \in \mathcal{F}_{\ell k}(\widehat{G}^{(k)})~ \mathrm{and} ~ D_f\cong D \}\right|t(D)}{\prod_{v \in V(D)}\mathrm{deg}_D^+(v)}}.
\end{align}

First recall how $D_f \cong D$ is constructed from $\ell k$ rooted hyperedges of $\widehat{G}^{(k)}$. In order to count the number of $f$ that give rise to a given $D$, note that we can permute the $k-1$ non-roots of each hyperedge (without changing $D_f$). We can also permute the hyperedges with the same root. If this root is a core vertex, then this however leads to the same $f$ because the permuted hyperedges are the same. If the root, $i$ say, is not a core vertex, then it occurs $\deg^+_{D^*}(i)$ times, but again, permuting the same hyperedges (with hyperedge $\{i,j\}^{(k)}$ with root $i$ occuring $m_{D^*}(i,j)$ times) leads to the same $f$. This implies that
\begin{align}\label{shizi8}
\left|\{f:f \in \mathcal{F}_{\ell k}(\widehat{G}^{(k)})~ \mathrm{and} ~ D_f\cong D \}\right|=((k-1)!)^{\ell k}\frac{\prod_{i \in V(\widehat{G})}(\deg^+_{D^*}(i))!}{\prod_{(i,j) \in E(D^*)}{m_{D^*}(i,j)!}}
\end{align}

From Lemma \ref{chongshuyinli}, it follows that if $v \in \mathcal{N}_{ij}$, then
$\deg_D^+(v)=(k-1)\frac{m_{D}(i,j)+m_{D}(j,i)}{2}$, which implies that
\begin{align}\label{shizi9}
&\prod_{v \in V(D)}{\mathrm{deg}_D^+(v)}%\notag \\
=\prod_{v \in V(D^*)}{\mathrm{deg}_D^+(v)}\prod_{v \in {V(D)\setminus V(D^*)}}{\mathrm{deg}_D^+(v)}\notag \\
&=\prod_{v \in V(D^*)}{(k-1)\mathrm{deg}_{D^*}^+(v)}\prod_{\{i,j\} \in E(\widehat{G})}{\left((k-1)\frac{m_{D^*}(i,j)+m_{D^*}(j,i)}{2}\right)^{k-2}}\notag \\
&=(k-1)^{|V(\widehat{G})|+(k-2)|E(\widehat{G})|}\prod_{v \in V(\widehat{G})}{\mathrm{deg}_{D^*}^+(v)}\prod_{\{i,j\} \in E(\widehat{G})}{\left(\frac{m_{D^*}(i,j)+m_{D^*}(j,i)}{2}\right)^{k-2}}.
\end{align}

Next, recall that there is a one-one correspondence between $\mathfrak{D}^*=\mathfrak{D}^*_d(\widehat{G}^{(k)})$ and $\mathfrak{D}_d(\widehat{G}^{(k)})$.
After substituting \eqref{scsshizi}, \eqref{shizi8} and  \eqref{shizi9} into \eqref{shizi7}, eliminating and collecting terms, we then obtain that
\begin{align*}
c_{\ell k}(\widehat{G}^{(k)})&=\ell \frac{k^{|E(\widehat{G})|(k-3)+|V(\widehat{G})|} 2^{|E(\widehat{G})|-|V(\widehat{G})|+1}}{(k-1)^{|V(\widehat{G})|+|E(\widehat{G})|(k-2)-1}} \sum_{D^* \in \mathfrak{D}^*} \frac{ \prod_{v \in V(D^*)}(\mathrm{deg}_{D^*}^+(v)-1)!t(D^*) }{\prod_{(i,j) \in E(D^*)}{m_{D^*}(i,j)!}}\\
&=\frac{2^{|E(\widehat{G})|-|V(\widehat{G})|} k^{|E(\widehat{G})|(k-3)+|V(\widehat{G})|}}{(k-1)^{|V(\widehat{G})|+|E(\widehat{G})|(k-2)-1}}\sum_{D^* \in \mathfrak{D}^* }\frac{|E(D^*)|}{b(D^*)}\prod_{v \in V(D^*)}(\mathrm{deg}_{D^*}^+(v)-1)!t(D^*),
\end{align*}
where in the second equality, we used that $|E(D^*)|=2\ell$ by Lemma \ref{jihedingli} and $b(D^*)$ equals the product of the factorials of the multiplicities of the edges.
By Lemma \ref{lem:parityclosedeulerian}, \eqref{czkpzk} now follows.
\end{proof}

From \eqref{shizi5}, \eqref{shizi1.1}, Lemmas \ref{jihedingli} and  \ref{pujuxishudingli}, we finally obtain our claimed expression for the spectral moment of $k$-power hypergraphs.

\begin{pro}\label{zydl}
Let $k \geq 3$. Then
\begin{align*}
\mathrm{S}_d(G^{(k)})={\mathcal{S}}_d(k).
\end{align*}
\end{pro}

\section{The characteristic polynomial of power hypergraphs}\label{sec:charpoly}

So far, we have used the ``trace formula" for tensors to derive the spectral moments of the power hypergraph $G^{(k)}$. From Lemma \ref{xiangyi}, we can find the eigenvalues of $G^{(k)}$ from the eigenvalues of the signed graphs on $G$. In this section, we will discuss how to derive the multiplicities of the eigenvalues using these eigenvalues, numbers of parity-closed walks in $G$, and spectral moments. The derived expressions will depend explicitly on $k$.

Let $\Sigma$ denote the set of squares of the nonzero eigenvalues of all signed subgraphs of $G$, let $\varsigma=|\Sigma|$, and denote the elements of $\Sigma$ by $\sigma^2_1,\sigma^2_2,\ldots,\sigma^2_{\varsigma}$.

By Lemma \ref{xiangyi} and the fact that the spectrum of $G^{(k)}$ is $k$-symmetric \cite{hu2014eigenvectors} (see also \cite[Theorem 3.1]{shao2015some}), we can write the characteristic polynomial $\phi_k(\lambda)$ of $G^{(k)}$ as
\begin{equation}\label{michaotuyinli}
\phi_k(\lambda)={\lambda^{\mu_0(k)}}\prod_{i=1}^{\varsigma}(\lambda^k-\sigma_i^2)^{\mu_i(k)},
\end{equation}
where $\mu_i(k)$ is the multiplicity of the corresponding eigenvalues $\lambda$.
Note that $\sigma_i^2 \in \Sigma$ does not necessarily give rise to an eigenvalue of $G^{(3)}$ if $\pm|\sigma_i|$ is not an eigenvalue of a signed {\em induced} subgraph. In that case $\mu_i(3)=0$.

Let
$$\mathrm{M}=\begin{bmatrix}
{\sigma^2_1}&{\sigma^2_2}& \cdots &{\sigma^2 _{\varsigma}}\\
{\sigma^4_1}&{\sigma^4_2}& \cdots &{\sigma^4 _{\varsigma}}\\
 \vdots & \vdots & \ddots & \vdots \\
{\sigma_1^{2\varsigma}}&{\sigma_2^{2\varsigma}}& \cdots &{\sigma_{\varsigma}^{2\varsigma }}
\end{bmatrix},$$
which is an invertible (cf.~Vandermonde) matrix of coefficients.
Let the multiplicity vector be $\mu(k)=\left(\mu_1 (k),\mu_2 (k),\ldots, \mu_{\varsigma} (k)\right)^{\top}$ and the spectral moments vector be $S(k)=\left({{\mathrm{S}_k}({G^{(k)}})},{{\mathrm{S}_{2k}}({G^{(k)}})}, \ldots,{{\mathrm{S}_{\varsigma k }}({G^{(k)}})}\right)^{\top}$.

From  \eqref{michaotuyinli}, the spectral moments can be written as $\mathrm{S}_{\ell k}(G^{(k)})=\sum_{i=1}^{\varsigma}{k\mu_i(k)\sigma^{2\ell}_i}$ for all positive integers $\ell$, which leads to the (nonsingular) system of equations
$$k\mathrm{M}\mu(k)=S(k)$$
for the multiplicities.

Next, we will rewrite $S(k)$ using the subgraph structural interpretation \eqref{shizi1.1} of the  spectral moments (Proposition \ref{zydl}).
Let $\chi=|\mathcal{G}(\varsigma)|$, denote the (non-isomorphic) subgraphs of $G$ with at most $\varsigma$ edges by $\widehat{G}_1, \widehat{G}_2, \ldots, \widehat{G}_{\chi}$ and let the parity-closed walk matrix be
\begin{align*}
\mathrm{P}=\begin{bmatrix}
{p_2(\widehat{G}_1)}&{p_2(\widehat{G}_2)}& \cdots &{p_2(\widehat{G}_{\chi})}\\
{p_{4}(\widehat{G}_1)}&{p_4(\widehat{G}_2)}& \cdots &{p_{4}(\widehat{G}_{\chi})}\\
 \vdots & \vdots & \ddots & \vdots \\
{p_{2\varsigma}(\widehat{G}_1)}&{p_{2\varsigma}(\widehat{G}_2)}& \cdots &{p_{2\varsigma}(\widehat{G}_{\chi})}
\end{bmatrix}.
\end{align*}
Moreover, consider the subgraph number vector $N=\left(N_G(\widehat{G}_1),N_G(\widehat{G}_2),\ldots,N_G(\widehat{G}_{\chi})\right)^{\top}$
and the diagonal $\chi \times \chi $ matrix $\mathrm{D}(k)$ with
$$(\mathrm{D}(k))_{ii}=
\frac{2^{|E(\widehat{G}_i)|-|V(\widehat{G}_i)|} k^{|E(\widehat{G}_i)|(k-3)+|V(\widehat{G}_i)|}}{(k-1)^{|V(\widehat{G}_i)|+|E(\widehat{G}_i)|(k-2)-1}}$$
for $i \in [\chi]$.
From Proposition \ref{zydl} and \eqref{shizi1.1}, with the additional remark that ${p_{2\ell}(\widehat{G_i})}=0$ when $|E(\widehat{G_i})| > \ell$ (i.e., $\mathrm{P}_{\ell i}=0$ when $\widehat{G_i} \notin \mathcal{G}(\ell)$), we then have
$${\mathrm{S}_{\ell k }}({G^{(k)}})=(k-1)^{|V|+(k-2)|E|-1}\sum_{i=1}^{\chi}\mathrm{D}(k)_{ii}\mathrm{P}_{\ell i}N_i,$$
and hence
$$\mu(k) =k^{-1}\mathrm{M}
^{-1}S(k)=\frac{(k-1)^{|V|+(k-2)|E|-1}}{k}\mathrm{M}^{-1}\mathrm{P}\mathrm{D}(k)N.$$

Thus, we have expressions for the multiplicities in $\mu(k)$. The remaining multiplicity $\mu_0(k)$ clearly follows from these, and we may conclude the following.

\begin{thm}\label{pro5.1}
Let $k \geq 3$.
Then the characteristic polynomial of $G^{(k)}$ is
\begin{equation*}
\phi_k(\lambda)={\lambda^{\mu_0(k)}}\prod_{i=1}^{\varsigma}(\lambda^k-\sigma_i^2)^{\mu_i(k)},
\end{equation*}
where $\mu_i(k)=\frac{(k-1)^{|V|+(k-2)|E|-1}}{k}(\mathrm{M}^{-1}\mathrm{P}\mathrm{D}(k)N)_i$ for $i \in [{\varsigma}]$ and $$\mu_0(k)=(|V|+(k-2)|E|)(k-1)^{|V|+(k-2)|E|-1}-k\sum_{i=1}^{\varsigma}\mu_i(k).$$
\end{thm}

\section{The geometric mean of the characteristic polynomials of all signed graphs with the same underlying graph}\label{sec4}

Using the expression in Theorem \ref{pro5.1}, we can extend the characteristic polynomial $\phi_k(\lambda)$ of a power hypergraph to $k=2$.
We thus define the pseudo-characteristic function
\begin{align*}
 \beta(\lambda)\equiv \phi_2(\lambda)={\lambda^{\mu_0(2)}}\prod_{i=1}^{\varsigma}(\lambda^2-\sigma_i^2)^{\mu_i(2)},
\end{align*}
where $\mu_i(2)=\frac{1}{2}(\mathrm{M}^{-1}\mathrm{P}N)_i$ for $i \in [{\varsigma}]$ (note that $\mathrm{D}(2)=I$) and $\mu_0(2)=|V|-2\sum_{i=1}^{\varsigma}\mu_i(2)$.
Note that $\beta(\lambda)$ is not a polynomial, in general, because the $\mu_i(2)$ are not necessarily integer for $i=0,1,\ldots,\varsigma$.

It is important to note that also for $k=2$, the ``spectral moments'' (of $\beta(\lambda)$) are determined by the equation $2\mathrm{M}\mu(2)=S(2)$. On the other hand, we defined $\mu(2)=\frac{1}{2}\mathrm{M}^{-1}\mathrm{P}N$, which implies by \eqref{eq:Pd=S2} that $$S(2)=\mathrm{P}N=(\mathcal{S}_2(2),\mathcal{S}_4(2),\ldots,\mathcal{S}_{2\varsigma}(2))^{\top}.$$
Thus, it follows that Proposition \ref{zydl} extends to $k=2$ in the sense that the spectral moment $\mathrm{S}_d(\beta)=\sum_{i=1}^{\varsigma}{2\mu_i(2)\sigma^{d}_i}$ of the pseudo-characteristic function equals ${\mathcal{S}}_d(2)$, for $d$ even and $d \leq 2\varsigma$ (and note that both equal $0$ for $d$ odd).

\begin{pro}\label{prop:Pdmoments}
Let $G$ be a graph and $d$ be even with $d \leq 2\varsigma$. Then the number of parity-closed walks in $G$ equals $P_d=\sum_{i=1}^{\varsigma}{2\mu_i(2)\sigma^{d}_i}$.
\end{pro}

Because $P_d=2^{-|E|}\sum_{\pi \in \Pi}{\mathrm{S}_d(G_{\pi})}$ by Theorem \ref{xindingli1}, we will next show that the multiplicity $\mu_i(2)$ is the average multiplicity of $\pm |\sigma_i|$ over all signed graphs on $G$, and obtain the following.

\begin{thm}\label{jieshao1}
The pseudo-characteristic function $\beta(\lambda)$ is the geometric mean of the characteristic polynomials of all signed graphs with underlying graph $G$,
i.e.,
\begin{align*}
  \beta(\lambda)=\prod_{ \pi \in \Pi} {\phi_{\pi}(\lambda)}^{2^{-|E|}}.
 \end{align*}
Moreover, $\beta(\lambda)$ is the characteristic polynomial of $G$ if and only if $G$ is a forest.
\end{thm}

\begin{proof}
Note first that $|\sigma_i|$ and $-|\sigma_i|$ have the same average multiplicity over all signed graphs as sign functions $\pi$ and $-\pi$ give opposite eigenvalues.
Let $\overline{\mu}_i$ then be the average multiplicity of $\pm |\sigma_i|$ over all signed graphs on $G$ and let $\overline{\mu}$ be the corresponding vector. Because $P_d=2^{-|E|}\sum_{\pi \in \Pi}{\mathrm{S}_d(G_{\pi})}$ by Theorem \ref{xindingli1}, it is clear that $\mathrm{P}N=2\mathrm{M}\overline{\mu}$. Because $2\mathrm{M}\mu(2)=S(2)=\mathrm{P}N$ and $M$ is invertible, it follows that $\overline{\mu}=\mu(2)$.

Now it follows that $\beta(\lambda)^{2^{|E|}}$ is a monic polynomial with the same spectral moments as $\prod_{ \pi \in \Pi} \phi_{\pi}(\lambda)$.
From Newton's identities, also known as the Girard-Newton formulae, it is known that a monic polynomial is determined by its spectral moments
\cite{eidswick1968proof,kalman2000matrix}, which finishes the proof of the main statement.

We are left to determine when $\beta(\lambda)$ is the characteristic polynomial $\phi(G)$ of $G$. Note first that if $G$ is a forest, then ${\phi_{\pi}(\lambda)}=\phi(G)$ for all $\pi \in \Pi$, and hence $\beta(\lambda)=\prod_{ \pi \in \Pi} {\phi_{\pi}(\lambda)}^{2^{-|E|}}=\phi(G)$.

On the other hand, if $\beta(\lambda)=\phi(G)$, then their spectral moments are the same, hence $P_d=S_d(G)$ for every $d$, which implies that every closed walk is a parity-closed walk. But then $G$ cannot have cycles, so it is a forest.

Thus, $\beta(\lambda)=\phi(G)$ if and only if $G$ is a forest.
\end{proof}

\begin{cor}
Let $G$ be a cycle. Then $\beta(\lambda)=\phi(\lambda^2-2)^{\frac{1}{2}}$.
\end{cor}

\begin{proof}
A signed cycle $(C_n, {\pi})$ is called positive (resp. negative) if the product of  signs of all edges of $(C_n, {\pi})$ is positive (resp. negative). Note that all positive (resp. negative) cycles are switching equivalent.
We use $\phi_+(\lambda)$ and  $\phi_-(\lambda)$ to denote the characteristic polynomial of a positive and negative cycle, respectively.
It is known that $\phi_+(\lambda) = \prod_{i=1}^{n} (\lambda-2\cos\frac{2i\pi}{n})$ \cite[p.~72]{cvetkovic1980spectra} and $\phi_-(\lambda) =\prod_{i=1}^{n} (\lambda-2\cos\frac{(2i-1)\pi}{n})$ \cite[Lemma 2.3]{fan2004largest}.
From Theorem \ref{jieshao1}, we have
\begin{align*}
  \beta^2(\lambda)&=\phi_+(\lambda)\phi_-(\lambda)
  =\prod_{i=1}^{2n}(\lambda-2\cos\tfrac{i\pi}{n})
  =\prod_{i=1}^{n}(\lambda^2-4\cos^2\tfrac{i\pi}{n}) \\
  &=\prod_{i=1}^{n}(\lambda^2-2-2\cos\tfrac{2i\pi}{n})=\phi(\lambda^2-2),
\end{align*}
where we have used the well-known identities $\cos\theta=-\cos(\theta-\pi)$ and $2\cos^2\theta=1+\cos2\theta$.
\end{proof}

For example, $\beta(\lambda)=(\lambda^2-1)(\lambda^2-4)^{\frac{1}{2}}$ for $G=C_3$, and note that this is not a polynomial.
In fact, when $G$ is connected, but not a tree, then $\beta(\lambda)$ is not a polynomial because the multiplicity of the spectral radius is not an integer. Indeed, note that if $G$ is connected, then its spectral radius has multiplicity $1$, and hence
$G_{\pi}$ has at most one eigenvalue $\rho(G)$ for all sign functions $\pi$.
If $G$ is not a tree, then there is a sign function $\pi$ such that $G_{\pi}$ is not switching equivalent to $G_{+}$  and hence, by Lemma \ref{yinlipubanjing}, $G_{\pi}$ does not have eigenvalue $\rho(G)$.
Thus, the average multiplicity of eigenvalue $\rho(G)$ over all sign functions is less than $1$, and hence $\beta(\lambda)$ is not a polynomial. In the next section, we will make this more precise and determine the multiplicity of $\rho(G)$ as a root of $\beta$.

From the arithmetic-geometric mean inequality, we can get an inequality between the matching polynomial $\alpha(\lambda)$ and the pseudo-characteristic function $\beta(\lambda)$ of $G$.

\begin{cor}\label{cor6.2}
Let $\lambda_0$ be a real number such that $\phi_{\pi}(\lambda_0) \geq 0$ for all $\pi \in \Pi$.  Then $\alpha(\lambda_0) \geq \beta(\lambda_0)$, with equality if and only if $\phi_{\pi}(\lambda_0)$ is constant over all $\pi \in \Pi$.
\end{cor}

Godsil and Gutman \cite{Godsil1978onthematching} showed that the matching polynomials $\alpha$ of some standard graphs can be expressed in terms of some classical orthogonal polynomials. For example $\alpha_{C_n}(\lambda)=2T_n(\lambda /2)$, where $T_n$ is a Chebyshev polynomial of the first kind. For the complete graphs $K_n$, they obtained that $\alpha_{K_n}(\lambda)={He}_n(\lambda)$, where ${He}_n$ denotes the probabilist's Hermite polynomial. Thus, Corollary \ref{cor6.2} implies that $\beta_{K_n}(\lambda_0)\leq {He}_n(\lambda_0)$, for example.

\section{The multiplicity of the spectral radius of a power hypergraph}\label{sec5}

Let $n_{\rho}(G^{(k)})$
be the total multiplicity of eigenvalues of the power hypergraph $G^{(k)}$ whose modulus is  equal to the spectral radius $\rho(G^{(k)})$.
We note that it follows from Theorem \ref{pro5.1} that this number is $k$ times the multiplicity of the spectral radius.

\begin{lem}\label{yinli1}
Let $k \geq 3$ and let $G$ be a connected graph. Then
\begin{align*}
  n_{\rho}(G^{(k)})=2^{|E|-|V|}k^{|E|(k-3)+|V|} \lim_{\ell\rightarrow \infty}{\frac{{P}_{2\ell}}{\rho(G)^{2\ell}}}.
\end{align*}
\end{lem}

\begin{proof}
It is known that the spectral radius of a power hypergraph $G^{(k)}$ is a positive number and $\rho(G^{(k)})^k=\rho(G)^2$
\cite{Zhou2014Some} (this follows also from Lemmas \ref{xiangyi} and \ref{yinlipubanjing}).
Since the $\ell k$-th order spectral moment $\mathrm{S}_{\ell k}(G^{(k)})$ is the sum of the $\ell k$-th powers of all eigenvalues of a power hypergraph $G^{(k)}$,
it follows that
\begin{align*}
  n_{\rho}(G^{(k)})=\lim_{\ell\rightarrow \infty}{\frac{\mathrm{S}_{\ell k}(G^{(k)})}{\rho(G^{(k)})^{\ell k}}}=\lim_{\ell\rightarrow \infty}{\frac{\mathrm{S}_{\ell k}(G^{(k)})}{\rho(G)^{2\ell}}}.
\end{align*}
By Proposition \ref{zydl}, we then have
\begin{align}\label{pubanjingshizi1}
  n_{\rho}(G^{(k)})=\lim_{\ell\rightarrow \infty}\frac{\sum_{\widehat{G} \in \mathcal{G}(\ell)}l_{G}(\widehat{G})p_{2\ell}(\widehat{G})}{\rho(G)^{2\ell}},
\end{align}
where $l_G(\widehat{G})={2^{|E(\widehat{G})|-|V(\widehat{G})|}{(k-1)^{|V(G^{(k)})|-|V(\widehat{G}^{(k)})|}} k^{|E(\widehat{G})|(k-3)+|V(\widehat{G})|}}{N_{G}(\widehat{G})}$.
We claim that in the summation over $\mathcal{G}(\ell)$, only $G$ itself contributes to the limit in \eqref{pubanjingshizi1}.
Indeed, for any graph ${H} \in \mathcal{G}(\ell)$ with $H \neq G$, let $\mathcal{H}(\ell)$ be the set of connected subgraphs (motifs) of the graph $H$ with at most $\ell$ edges.
By  \eqref{pubanjingshizi1}, we have
$$n_{\rho}(H^{(k)})=\lim_{\ell\rightarrow \infty}\frac{\sum_{\widehat{G} \in \mathcal{H}(\ell)}l_{H}(\widehat{G})p_{2\ell}(\widehat{G})}{\rho(H)^{2\ell}},$$
which implies that
$$
\lim_{\ell\rightarrow \infty}\frac{\sum_{\widehat{G} \in \mathcal{H}(\ell)}l_{H}(\widehat{G})p_{2\ell}(\widehat{G})}{\rho(G)^{2\ell}}=0,
$$
because $\rho(H)<\rho(G)$.

Note that $l_{H}(\widehat{G}){p_{2\ell}(\widehat{G})} \geq 0$ and $l_{H}(\widehat{G})>0$ for any ${\widehat{G} \in \mathcal{H}(\ell)}$ and any $H \in \mathcal{G}(\ell)$  ($H \neq G$), hence it follows that
\begin{align*}
\lim_{\ell\rightarrow \infty}\frac{p_{2\ell}(\widehat{G})}{\rho(G)^{2\ell}}=0
\end{align*}
for any ${\widehat{G} \in \mathcal{G}(\ell)}$ with ${\widehat{G} \neq G}$.
By \eqref{pubanjingshizi1}, we thus have
\begin{align*}
 n_{\rho}(G^{(k)})&={2^{|E|-|V|} k^{|E|(k-3)+|V|}}\lim_{\ell\rightarrow \infty}\frac{p_{2\ell}({G})}{\rho(G)^{2\ell}}\\
&={2^{|E|-|V|} k^{|E|(k-3)+|V|}}\lim_{\ell\rightarrow \infty}\frac{\sum_{\widehat{G} \in \mathcal{G}(\ell)}p_{2\ell}(\widehat{G}){N_{G}(\widehat{G})}}{\rho(G)^{2\ell}}\\
&=2^{|E|-|V|}k^{|E|(k-3)+|V|} \lim_{\ell\rightarrow \infty}{\frac{{P}_{2\ell}}{\rho(G)^{2\ell}}}.
\end{align*}

\end{proof}

From Theorem \ref{xindingli1}, we have that ${{P}_{2\ell}}=2^{-|E|}\sum_{\pi \in \Pi}{\mathrm{S}_{2\ell}(G_{\pi})}$.
Using Lemma \ref{yinlipubanjing} about the spectral radius of signed graphs, we obtain $\lim_{\ell\rightarrow \infty}{\frac{{P}_{2\ell}}{\rho(G)^{2\ell}}}$ and determine the multiplicity of the spectral radius of $G^{(k)}$.

\begin{thm}\label{jieshao2}
For a connected graph $G$ and $k\geq 3$,
the multiplicity of the spectral radius of $G^{(k)}$ is $k^{|E|(k-3)+|V|-1}$.
\end{thm}
\begin{proof}
From  Theorem \ref{xindingli1} and Lemma \ref{yinli1}, we have that
\begin{align}\label{shizichongshu}
  n_{\rho}(G^{(k)})&=2^{-|V|}k^{|E|(k-3)+|V|} \lim_{\ell\rightarrow \infty}{\frac{\sum_{\pi \in \Pi}{\mathrm{S}_d(G_{\pi})}}{\rho(G)^{2\ell}}} \notag\\
  &=2^{-|V|}k^{|E|(k-3)+|V|} \lim_{\ell\rightarrow \infty}{\frac{\sum_{\pi \in \Pi}{\sum_{\lambda_{\pi} \in \sigma(G_{\pi})}\lambda^{2\ell}_{\pi}}}{\rho(G)^{2\ell}}},
\end{align}
where $\sigma(G_{\pi})$ is the spectrum of $G_{\pi}$.

Let $\mathcal{D}$ be the set of all $|V|\times |V|$ diagonal matrices with diagonal entries $\pm 1$.
Let $\Pi_{+}$ and $\Pi_{-}$ denote the sets of sign functions $\pi \in \Pi$ such that $A(G_{\pi})=D^{-1}A(G)D$ and $\Pi_{+}=D^{-1}\left(-A(G)\right)D$ for some $D \in \mathcal{D}$, respectively. By Lemma \ref{yinlipubanjing}, $\Pi_{+}$ contains the sign functions for which $G_{\pi}$ has an eigenvalue $\rho(G)$, and $\Pi_{-}$ contains the sign functions for which $G_{\pi}$ has an eigenvalue $-\rho(G)$.

We observe that since $G$ is connected, $D^{-1}A(G)D=A(G)$ implies that $D=I$ or $D=-I$. This implies that there is a two-one correspondence between $\mathcal{D}$ and $\Pi_+$ and similarly between $\mathcal{D}$ and $\Pi_-$. Thus,
\begin{align*}
  |\Pi_+|= |\Pi_-|=\frac{|\mathcal{D}|}{2}=2^{|V|-1}.
\end{align*}
From  Lemma \ref{yinlipubanjing}, we now have that
\begin{align*}
 \lim_{\ell\rightarrow \infty}{\frac{\sum_{\pi \in \Pi}{\sum_{\lambda_{\pi} \in \sigma(G_{\pi})}\lambda^{2\ell}_{\pi}}}{\rho(G)^{2\ell}}}&=\lim_{\ell\rightarrow \infty}{\frac{\sum_{\pi \in \Pi_+}\rho(G_{\pi})^{2\ell}+\sum_{\pi \in \Pi_-}(-\rho(G_{\pi}))^{2\ell}}{\rho(G)^{2\ell}}}\\
&=|\Pi_+|+|\Pi_-|
=2^{|V|}
\end{align*}
(note that it is not relevant whether $G$ is bipartite, in which case $\Pi_+=\Pi_-$, or not).
By \eqref{shizichongshu}, we thus have $n_{\rho}(G^{(k)})=k^{|E|(k-3)+|V|}$, and hence the multiplicity of the spectral radius of the power hypergraph $G^{(k)}$ is $k^{|E|(k-3)+|V|-1}$.
\end{proof}

As a corollary of the proof, we can also extend the above result to the case $k=2$. Indeed, as $\Pi_+$ contains all sign functions $\pi$ such that $G_{\pi}$ has an eigenvalue $\rho(G)$, each with multiplicity $1$, and the multiplicity of the spectral radius $\rho(G)$ as a root of $\beta$ is the average multiplicity over all sign functions, it follows that this multiplicity equals $2^{-|E|}|\Pi_+|$.

%We note that the arguments in the above proof also apply to the case that $k=2$, in fact, we do not even need an analogue for Lemma \ref{yinli1}. Indeed,
%
%$$n_{\rho}(\beta)=\lim_{\ell\rightarrow \infty}\frac{\mathcal{S}_{2\ell}(2)}{\rho(G)^{2\ell}}=
%2^{-|E|} \lim_{\ell\rightarrow \infty}\frac{\sum_{\pi \in \Pi}{\mathrm{S}_d(G_{\pi})}}{\rho(G)^{2\ell}}=2^{|V|-|E|},$$
%and hence we have the following for the multiplicity of the spectral radius $\rho(G)$ as a root of $\beta$.

\begin{pro}\label{pro:multbeta}
For a connected graph $G$,
the multiplicity of the spectral radius $\rho(G)$ as a root of $\beta$ is $2^{-|E|+|V|-1}$.
\end{pro}

Consequently, we confirm our earlier observation that for a connected graph, $\beta$ is a polynomial only for a tree.

\section*{Acknowledgement}
This work is supported by the National Natural Science Foundation of China (No. 12071097), the Natural Science Foundation for The Excellent Youth Scholars of the Heilongjiang Province (No. YQ2022A002) and the Fundamental Research Funds for the Central Universities.

\section*{References}
\bibliographystyle{plain}
\bibliography{spbib}

\begin{thebibliography}{10}

\bibitem{brualdi1983determinantal}
Richard Brualdi and Hans Schneider.
\newblock Determinantal identities: Gauss, {S}chur, {C}auchy, {S}ylvester,
  {K}ronecker, {J}acobi, {B}inet, {L}aplace, {M}uir, and {C}ayley.
\newblock {\em Linear Algebra and its Applications}, 52:769--791, 1983.

\bibitem{chang2008perron}
Kungching Chang, Kelly Pearson, and Tan Zhang.
\newblock Perron-{F}robenius theorem for nonnegative tensors.
\newblock {\em Communications in Mathematical Sciences}, 6(2):507--520, 2008.

\bibitem{Chen2020Spectral}
Lixiang Chen, Changjiang Bu, and Jiang Zhou.
\newblock Spectral moments of hypertrees and their applications.
\newblock {\em Linear and Multilinear Algebra}, 2021.
\newblock Doi:{10.1080/03081087.2021.1953431}.

\bibitem{Chen2022All}
Lixiang Chen, Edwin~R. van Dam, and Changjiang Bu.
\newblock All eigenvalues of the power hypergraph and signed subgraphs of a
  graph.
\newblock arXiv:2209.03709, 2022.

\bibitem{CLARK20211}
Gregory Clark and Joshua Cooper.
\newblock A {H}arary-{S}achs theorem for hypergraphs.
\newblock {\em Journal of Combinatorial Theory, Series B}, 149:1--15, 2021.

\bibitem{cooper2012spectra}
Joshua Cooper and Aaron Dutle.
\newblock Spectra of uniform hypergraphs.
\newblock {\em Linear Algebra and its Applications}, 436(9):3268--3292, 2012.

\bibitem{crabtree1966applications}
Douglas Crabtree.
\newblock Applications of {M}-matrices to non-negative matrices.
\newblock {\em Duke Mathematical Journal}, 33(1):197--208, 1966.

\bibitem{cvetkovic1980spectra}
Drago{\v{s}} Cvetkovi{\'c}, Horst Sachs, and Michael Doob.
\newblock {\em Spectra of Graphs: Theory and Applications}.
\newblock VEB Deutscher Verlag der Wissenschaften, 1980.

\bibitem{duval2009simplicial}
Art Duval, Caroline Klivans, and Jeremy Martin.
\newblock Simplicial matrix-tree theorems.
\newblock {\em Transactions of the American Mathematical Society},
  361(11):6073--6114, 2009.

\bibitem{eidswick1968proof}
J.A. Eidswick.
\newblock A proof of {N}ewton's power sum formulas.
\newblock {\em American Mathematical Monthly}, 75(4):396--396, 1968.

\bibitem{fan1960note}
Ky~Fan.
\newblock Note on {M}-matrices.
\newblock {\em The Quarterly Journal of Mathematics}, 11(1):43--49, 1960.

\bibitem{fan2004largest}
Yizheng Fan.
\newblock Largest eigenvalue of a unicyclic mixed graph.
\newblock {\em Applied Mathematics-A Journal of Chinese Universities Series B},
  19(2):140--148, 2004.

\bibitem{matthew2016multi}
Matthew Farrell and Lionel Levine.
\newblock Multi-{E}ulerian tours of directed graphs.
\newblock {\em The Electronic Journal of Combinatorics}, 23(2), 2016.
\newblock Article ID \#P2.21.

\bibitem{Godsil1978onthematching}
Christopher Godsil and Ivan Gutman.
\newblock On the matching polynomial of a graph.
\newblock In L.~Lov\'{a}sz and V.T. S\'{o}s, editors, {\em Algebraic Methods in
  Graph Theory}, volume~I of {\em Colloquia Mathematica Societatis J\'{a}nos
  Bolyai}, pages 241--249, North-Holland, New York, August 1978. the Conderence
  held in Szged.
\newblock papers from the Conderence held in Szged, August 24-31,1978.

\bibitem{griffing2014structural}
Alexander Griffing, Benjamin Lynch, and Eric Stone.
\newblock Structural properties of the minimum cut of partially-supplied
  graphs.
\newblock {\em Discrete Applied Mathematics}, 177:152--157, 2014.

\bibitem{hu2013determinants}
Shenglong Hu, Zhenghai Huang, Chen Ling, and Liqun Qi.
\newblock On determinants and eigenvalue theory of tensors.
\newblock {\em Journal of Symbolic Computation}, 50:508--531, 2013.

\bibitem{hu2014eigenvectors}
Shenglong Hu and Liqun Qi.
\newblock The eigenvectors associated with the zero eigenvalues of the
  {L}aplacian and signless {L}aplacian tensors of a uniform hypergraph.
\newblock {\em Discrete Applied Mathematics}, 169:140--151, 2014.

\bibitem{Hu2013Cored}
Shenglong Hu, Liqun Qi, and Jiayu Shao.
\newblock Cored hypergraphs, power hypergraphs and their {L}aplacian
  {H}-eigenvalues.
\newblock {\em Linear Algebra and its Applications}, 439:2980--2998, 2013.

\bibitem{kalman2000matrix}
Dan Kalman.
\newblock A matrix proof of {N}ewton's identities.
\newblock {\em Mathematics Magazine}, 73(4):313--315, 2000.

\bibitem{lim2005singular}
Lek-Heng Lim.
\newblock Singular values and eigenvalues of tensors: a variational approach.
\newblock In {\em 1st IEEE International Workshop on Computational Advances in
  Multi-Sensor Adaptive Processing}, pages 129--132. IEEE, 2005.

\bibitem{Marcus2015Interlacing}
Adam Marcus, Daniel Spielman, and Nikhil Srivastava.
\newblock Interlacing families {I}: Bipartite {R}amanujan graphs of all
  degrees.
\newblock {\em Annals of Mathematics}, 182(1):307--325, 2015.

\bibitem{qi2005eigenvalues}
Liqun Qi.
\newblock Eigenvalues of a real supersymmetric tensor.
\newblock {\em Journal of Symbolic Computation}, 40(6):1302--1324, 2005.

\bibitem{shao2015some}
Jiayu Shao, Liqun Qi, and Shenglong Hu.
\newblock Some new trace formulas of tensors with applications in spectral
  hypergraph theory.
\newblock {\em Linear and Multilinear Algebra}, 63(5):971--992, 2015.

\bibitem{stanic2019bounding}
Zoran Stani{\'c}.
\newblock Bounding the largest eigenvalue of signed graphs.
\newblock {\em Linear Algebra and its Applications}, 573:80--89, 2019.

\bibitem{zhou2021enumeration}
Jiang Zhou and Changjiang Bu.
\newblock The enumeration of spanning tree of weighted graphs.
\newblock {\em Journal of Algebraic Combinatorics}, 54(1):75--108, 2021.

\bibitem{Zhou2014Some}
Jiang Zhou, Lizhu Sun, Wenzhe Wang, and Changjiang Bu.
\newblock Some spectral properties of uniform hypergraph.
\newblock {\em Electronic Journal of Combinatorics}, 21:4--24, 2014.

\end{thebibliography}
%\begin{thebibliography}{}
%\bibitem{R.}
%\end{thebibliography}
\end{spacing}
\end{document}